\documentclass[11pt]{article}
\usepackage{amssymb, amsmath, amsthm}
\usepackage{graphicx}
\usepackage{color}
\newtheorem{thm}{Theorem}[section]
\newtheorem{lem}{Lemma}[section]
\newtheorem{cor}[lem]{Corollary}

\theoremstyle{definition}

\theoremstyle{remark}

\numberwithin{equation}{section}

\newcommand{\E}{\mathbf{E}\,}

\newcommand{\Tr}{\mathrm{Tr}\;\!}

\newcommand{\re}{\mathrm{Re}\;\!}
\newcommand{\im}{\mathrm{Im}\;\!}

\newenvironment{Proof of}{\removelastskip\par\medskip
\noindent{\em Proof of} \rm}{\penalty-20\null\hfill$\square$\par\medbreak}

\begin{document}
\setcounter{page}{1}

\title{\bf The rate of convergence of spectra of sample covariance matrices}

\author{{\bf F. G\"otze}\\{\small Faculty of Mathematics}
\\{\small University of Bielefeld}\\{\small Germany}
\and {\bf A. Tikhomirov}$^{1}$\\{\small Faculty of Mathematics and Mechanics}\\{\small Sankt-Peterburg State University}
\\{\small S.-Peterburg, Russia}}
\date{}
\maketitle
 \footnote{$^1$Partially supported by RFBF
grant N
07-01-00583-a, by RF grant of the leading scientific schools NSh-4222.2006.1. Partially supported by CRC 701 ``Spectral Structures and Topological Methods in Mathematics'', Bielefeld}


\today

\begin{abstract}It is shown that the  Kolmogorov distance
between the spectral distribution function
of a random covariance matrix  $\frac1p XX^T$,
where $X$ is a $n\times p$ matrix with independent entries  
and the distribution function 
of the  Marchenko-Pastur law is of order $O(n^{-1/2})$. The bounds hold {\it uniformly} for any $p$, including  $\frac pn$  equal or close to
$1$.  

\end{abstract}

\maketitle
\markboth{ F. G\"otze, A.Tikhomirov}{On the circular law}

\section{Introduction}
Let $X_{ij},{  }1\le i\le p, 1\le j\le n$, be independent random variables 
with $\E X_{ij}=0$ and $\E X_{ij}^2=1$ and 
$\bold X_p=\Big(X_{ij}\Big)_{\{1\le i\le p,{  } 1\le j\le n\}}$.
Denote by $\lambda_1\le\ldots\le \lambda_p$ the eigenvalues of 
the symmetric matrix 
$$
\bold W:=\bold W_p:=\frac1n \bold X
_p\bold X_p^T
$$
and define its empirical distribution by
$$
F_p(x)=\frac1p\sum_{k=1}^pI_{\{\lambda_k\le x\}},
$$
where $I_{\{B\}}$ denotes the indicator of an event $B$. 
We shall investigate the rate 
of convergence of the expected spectral distribution $\E F_p(x)$ as well as
$F_p(x)$ to
the Marchenko-Pastur distribution function $F_y(x)$ with  density
$$
f_y(x)=\frac1{2xy\pi}\sqrt{(b-x)(x-a)}I_{\{[a,b]\}}(x)+
I_{\{[1,\infty)\}}(y)(1-y^{-1})\delta(x),
$$
where $y\in (0,\infty)$ and $a=(1-\sqrt y)^2$, $b=(1+\sqrt y)^2$.
Here we denote by $\delta(x)$ the Dirac delta-function 
and by $I_{\{[a,b]\}}(x)$
the indicator function of the interval  $[a,b]$.
As in Marchenko and Pastur \cite{M-P:67}
and Pastur \cite{P:73} assume that $X_{ij}$, $i,j\ge1$,
are independent identically distributed random variables such that 
$$
\E X_{ij}=0,\qquad \E X_{ij}^2=1 \quad
\text{ and  }\qquad \E |X_{ij}|^4<\infty,\qquad\text{ for all } i,j .
$$
Then $\E F_p\to F_y$ and $F_p\to F_y$ in probability, 
where \newline $y=\lim_{n\to\infty}y_p:=\lim_{n\to\infty}
(\frac p n)\in(0,\infty)$.

Let $y:=y_p:=p/n$.
We introduce the following distance between the distributions 
$\E F_p(x)$ and $F_{y}(x)$
$$
\Delta_p:=\sup_x|\E F_p(x)-F_{y}(x)|
$$
as well as another distance between the distributions $F_p(x)$ and $F_{y}(x)$
$$
\Delta_p^*:=\E \sup_x|F_p(x)-F_{y}(x)|.
$$
We shall use the notation $\xi_n=O_P(a_n)$  if,
 for any $\varepsilon>0$, there exists an $L>0$ such that

 $\Pr\{|\xi_n| \ge L a_n\}\le  \varepsilon$.
 Note that, for any $L>0$,
 $$
 \Pr\{\sup_x|F_p(x)-F_{y}(x)|\ge L\}\le \frac{\Delta_p^*}{L}.
 $$
Hence  bounds for
$\Delta_p^*$ provide bounds for the rate of convergence in probability 
of  the quantity
$\sup_x|F_p(x)-F_{y}(x)|$  to zero. 
Using our techniques it is straightforward though technical  to prove that the
rate of almost sure convergence is at least $O(n^{-1/2+\epsilon})$, for any $
\epsilon >0$. In view of the length of the  proofs for the results stated above
we  refrain from including those  details in this paper as well.

Bai \cite{Bai:93} proved that 
$\Delta_p=O(n^{-\frac14})$,
assuming $\E X_{ij}=0$,
$\E X_{ij}^2=1$,\newline
$\sup\limits_n \sup\limits_{i,j}
\E X_{ij}^4\bold I_{\{|X_{ij}|>M\}}\to 0,\quad\text{as  } 
M\to\infty$,
and 
$$
y\in(\theta,\Theta)\text{   such that }0<\theta<\Theta<1
\text{  or  }
1<\theta<\Theta<\infty.
$$
If $y$ is close to $1$ the limit density and the Stieltjes transform of 
the limit density have a singularity.
In this case the investigation of the rate of convergence is more difficult.
Bai \cite{Bai:93} has shown that, if $0<\theta\le y_p\le \Theta<\infty$,
$
\Delta_p=O(n^{-\frac5{48}})
$. Recently Bai et al. \cite{Bai:03} have shown for $y_p$ equal to $1$
or asymptotically near $1$ that $\Delta_p=O(n^{-\frac 1 {8}})$ (see also \cite{Bai:05}). It is clear that
the case $y_p \approx 1$ requires different techniques. Results of the 
authors \cite{GT:05} show that for Gaussian r.v. $X_{ij}$ actually the rate $\Delta_p
=O(n^{- 1  })$ is the correct rate of approximation including the case $y=1$.

%
By 
$C$ (with an index or without it) we shall denote generic
absolute constants,
whereas $C(\,\cdot\,,\,\cdot\,)$ will denote
positive constants depending on arguments.
Introduce the notation, for $k\ge 1$,
$$
M_k:=M_k^{(n)}:=\sup_{1\le j,k\le n}\E|X_{jk}|^k.
$$
Our main results are the following
\begin{thm} \label{thm1.1} Let $1\ge y>\theta>0$, for some positive constant $\theta$.
 Assume that $\E X_{jk}=0$, $\E|X_{jk}|^2=1$, and
\begin{equation}
M_4:=\sup_{1\le j,k\le n}\E|X_{jk}|^4<\infty.
\end{equation}
Then there exists a positive constant $C(\theta)>0$ depending on $\theta$ 
such that 
$$
\Delta_p\le C(\theta)\,M_4^{\frac12} n^{-1/2}.
$$
\end{thm}
\begin{thm} \label{thm1.2} Let $1\ge y>\theta>0$, for some positive constant $\theta$. Assume that $X_{ij}$
$\E X_{jk}=0$, $\E|X_{jk}|^2=1$, and
$$
M_{12}:=\sup_{1\le j,k\le n}\E|X_{jk}|^{12}<\infty.
$$
Then there exists a  positive constant $C(\theta)>0$ depending on $\theta$ 
such that 
$$
\Delta_p^*=\E \sup_x|F_p(x)-G(x)|
\le C(\theta)M_{12}^{\frac16}\, n^{-1/2}.
$$
\end{thm}

We shall prove the same result for the following class of sparse matrices. Let
$\varepsilon_{jk}$, $j=1,\ldots,n$, $k=1,\ldots,p$, denote Bernoulli random variables which are independent in aggregate and
independent of $(X_{jk})$  with
$p_n:=\Pr\{\varepsilon_{jk}=1\}$. Consider the matrix $\bold
X^{(\varepsilon)}=\frac1{\sqrt{np_n}}(\varepsilon_{jk}X_{jk})$.
Let $\lambda_1^{(\varepsilon)},\ldots,\lambda_p^{(\varepsilon)}$ denote the (complex)
eigenvalues of the matrix $\bold X^{(\varepsilon)}$ and denote by
$F_p^{(\varepsilon)}(x)$ the empirical spectral distribution
function of the matrix $\bold X^{(\varepsilon)}$, i. e.
\begin{equation}
F_p^{(\varepsilon)}(x):=\frac1p\sum_{j=1}^pI_{\{\lambda_j^{(\varepsilon)}\leq x,\}}.
\end{equation}
%
%
\begin{thm}\label{sparse}
Let $X_{jk}$ be independent random variables with
\begin{equation}\notag
\E X_{jk}=0,\qquad \E |X_{jk}|^2=1,\quad \text{ and}\quad \E |X_{jk}|^{4}.
\end{equation}
Assume that $np_n\to\infty$ as $n\to\infty$
Then  
\begin{equation}
 \Delta_n^{(\varepsilon)}:=\sup_x|\E F_p^{(\varepsilon)}(x)-F_p(x)|\le CM_4^{1/2}(np_n)^{-\frac12}.
\end{equation}
.
\end{thm}
We have developed a new approach to the investigation of convergence of spectra of sample covariance matrices based on the so-called Hadamar matrices. Note that our approach allows us to obtain a bound of the rate of convergence to the Marchenko-Pastur distribution uniformly in $1\ge y\ge\theta$
(including $y=1$). In this paper we give the proof of Theorem \ref{thm1.1} only.
To prove Theorem \ref{thm1.2} and \ref{sparse} it is enough to repeat the proof of Theorem 1.2 and Corollary 1.3 in \cite{GT:03} with   inessential   changes.

\section{  Inequalities for the  distance between distributions
via Stieltjes transforms.}
We define the Stieltjes transform $s(z)$ 
of a random variables $\xi$ with the distribution function $F(x)$
(the Stieltjes transform $s(z)$ of distribution function $F(x)$)
$$
s(z):=\E\frac1{\xi-z}=\int_{-\infty}^{\infty}\frac1{x-z}d\,F(x),
\quad z=u+iv,\quad  v>0.
$$

\begin{lem} Let $F$ and $G$ be a distribution functions such that
 \begin{equation}
\int_{-\infty}^{\infty}|F(x)-G(x)|\,dx<\infty.
\end{equation}
Denote their Stieltjes 
transforms by $s(z)$ and $t(z)$ respectively.
Assume that the distribution $G(x)$ has  support contained in the bounded interval $I=[a,b]$.
Assume that there exists a positive constant $c_g$ such that
\begin{equation}
 \sup_x\frac d{dx}G(x)\le c_g.
\end{equation}
Denote their Stieltjes 
transforms by $s(z)$ and $t(z)$ respectively. Let $v>0$. 
Then there exist some constants $C_1(c_g),\, C_2(c_g),\, C_3(c_g)$ 
depending only on $c_g$, such that
\begin{align}
\Delta(F,G)&:=\sup_{x}\, |F(x)-G(x)| \\&\le \, 
C_1\, \sup_{x\in I}
\, |\text{\rm{Im}}\,\Big(\int_{-\infty}^x(s(z)-s_y(z))\, du\Big)|+
C_2\, v ,
\end{align}
where $z=u+iv$.
\end{lem}
A proof of Lemma 2.1  in G\"otze, Tikhomirov  \cite{GT:03},  . 
\begin{cor} The following inequality holds, for any $0<v<V$,
 \begin{align}
\Delta(F,G)\le &C_1\int_{-\infty}^{\infty}|(s(u+iV)-t(u+iV))\, |du+
C_2\, v\\&
+C_1\sup_{x\in I}
\left|{\re}\,\left\{\int_{v}^V (s(x+iu)-t(x+iu))du\right\}\right|.
\end{align}
\end{cor}

\section{ The main Lemma}
Let $\xi\ge 0$ be a positive random variables with distribution function $F(x)$.
Let $\varkappa$ be a Rademacher random variable with value $\pm 1$ with porbability $1/2$.
 Consider a random variable $\widetilde\xi:=\varkappa\xi$ and denote its distribution function by $\widetilde F(x)$.
For any $x$, we have
\begin{equation}
 \widetilde F(x)=\frac12(1+\text{\rm sgn}x\,F(x^2))
\end{equation}
This equality implies that
\begin{equation}
 \widetilde p(x):=\frac d{dx}\widetilde F(x)=|x|p(x),
\end{equation}
where
\begin{equation}
 p(x)=\frac d{dx}F(x).
\end{equation}
For the Marchenko--Pastur distribution with parameter $y\in(0,1]$, we have
\begin{equation}
 \widetilde p_y(x)=|x|p_y(x)=\frac1{2\pi y|x|}\sqrt{(x^2-a)(b-x^2)}.
\end{equation}
It is straighforward to check that, for $y\in(0,1]$,
\begin{equation}\label{density}
 \sup_x\widetilde p_y(x)\le \frac1{\pi \sqrt y(1+\sqrt y)}.
\end{equation}
   Note also that the distribution $\widetilde F_y(x)$ has  a support which is contained in the union of the
intervals $[-(1+\sqrt y),-(1-\sqrt y)]\cup[(1-\sqrt y),(1+\sqrt y)]$.

 Introduce the following matrix
\begin{equation}
 \bold H:=\left(\begin{matrix}{\bold O\quad\bold X}\\{\bold X^*\quad\bold O}\end{matrix}\right),
\end{equation}
where $\bold O$ is the matrix with  zero entries only.
Consider the resolvent matrix
\begin{equation}
 \bold R(z)=(\bold H-z\bold I)^{-1},
\end{equation}
where $\bold I$ denotes the identity matrix of order $n+p$.

Let $s_y(z)$ denote the Stieltjes transform of the Marchenko--Pastur distribution function  with parameter $y$. Denote by $\widetilde s_y(z)$ the Stieltjes transform of the distribution function $\widetilde F_y(x)$. It is straighforward to check that
\begin{equation}
 \widetilde s_y(z)=zs_y(z^2).
\end{equation}
For the Stieltjes transform of the expected spectral distribution function of the sample covariance matrix $s_p(z)$
and its ``symmetrization'' $\widetilde s_p(z)$ we have, 
\begin{equation}
 \widetilde s_p(z)=zs_p(z^2).
\end{equation}
From the equation for $s_y(z)$ 
\begin{equation}
 s_y(z)=-\frac1{z+y-1+yzs_y(z)}
\end{equation}
it follows that
\begin{equation}
 \widetilde s_y(z)=-\frac1{z+y\widetilde s_y(z)+\frac{y-1}z}.
\end{equation}
By inversion of the partitioned matrix formula (see \cite{Horn}, p. 18, Section 0.7.3) , we have
\begin{equation}
 \bold R(z)=\left(\begin{matrix}{z(\bold X\bold X^*-z^2\bold I_n)^{-1}\quad \bold X(\bold X^*\bold X-z^2\bold I_p)^{-1}}\\{(\bold X^*\bold X-z^2\bold I_p)^{-1}\bold X^*\quad (\bold X^*\bold X-z^2\bold I_p)^{-1}}\end{matrix}\right)
\end{equation}

This equality implies that
\begin{equation}
 \widetilde s_p(z)=\frac1n\sum_{j=1}^n\E R_{jj}(z)=\frac1{n}\sum_{j=1}^pR_{j+n,j+n}(z)+\frac{y-1}{z}
\end{equation}
and
\begin{equation}
 \frac1{p}\sum_{j=1}^pR_{j+n,j+n}(z)=y\frac1n\sum_{j=1}^nR_{j,j}(z)+\frac{1-y}{z}.
\end{equation}

Tfor the  readers convenient we state here two  Lemmas, which follow from Shur's complement formula
(see, for example, \cite{GT:03}).
Let $\bold A=\Big(a_{kj}\Big)$ denote a 
matrix 
of order $n$ 
and $\bold A_k$ denote the principal sub-matrix of  order $n-1$, i.e. 
$\bold A_k$ is obtained from $\bold A$ by deleting the $k$-th row and the $k$-th
column. Let $\bold A^{-1}=\Big(a^{jk}\Big)$.  Let $\bold a_k'$ 
denote the vector 
obtained from the $k$-th row of $\bold A$ by deleting the $k$-th entry and 
$\bold b_k$ the vector from the $k$-th column by deleting the $k$-th entry.
Let  $\bold I$ with  subindex or without  denote the  identity matrix
of  corresponding size.

\begin{lem}\label{Lemma 3.1}
Assume that $\bold A$ and $\bold A_k$ are
nonsingular. Then we have
$$
a^{kk}=\frac1{a_{kk}-\bold a'_k\bold A_k^{-1}\bold b_k}.
$$
\end{lem}
\begin{lem}\label {Lemma 3.2} Let $z=u+iv$, and $\bold A$ be an $n\times n$
symmetric matrix. Then 
\begin{align}
\Tr (\bold A-z\bold I_n)^{-1}-
\Tr (\bold A_k-z\bold I_{n-1})^{-1}
&=\frac{1+\bold a_k'(\bold A_k-z\bold I_{n-1})^{-2}\bold a_k}{a_{kk}-z-
\bold a_k'(\bold A_k-z\bold I_{n-1})^{-1}\bold a_k}\notag\\&=(1+\bold a_k'(\bold A_k-z\bold I_{n-1})^{-2}\bold a_k)\,a^{kk}.
\end{align}
and 
$$
\Big|\Tr \Big(\bold A-z\bold I_n\Big)^{-1}-\Tr\Big(\bold A_k-z\bold I_{n-1}\Big)
^{-1}\Big|
\le v^{-1}.
$$
\end{lem}
Applying Lemma \ref{Lemma 3.1} with $\bold A=\bold W$ 
  we may write, for $j=1,\ldots,n$
\begin{align}\label{ineq3.1}
R_{j,j}&=-\frac1{z+y\widetilde s_p(z)+\frac{y-1}z-\varepsilon_j}=-\frac1{z+y\widetilde s_p(z)+\frac{y-1}z}\notag\\&+
\frac{\varepsilon_j}{(z+y\widetilde s_p(z)+\frac{y-1}z)(z+y\widetilde s_p(z)+\frac{y-1}z-\varepsilon_j)}\notag\\&=
-\frac1{z+y\widetilde s_p(z)+\frac{y-1}z}\left(1-\varepsilon_jR_{j,j}\right),
\end{align}
where
\begin{equation}\label{rep5.1}
 \varepsilon_j=\varepsilon_j^{(1)}+\varepsilon_j^{(2)}+\varepsilon_j^{(3)}+\varepsilon_j^{(4)}
\end{equation}
with
\begin{align}
 \varepsilon_j^{(1)}&=\frac1p\sum_{1\le k\ne l\le p}X_{jk}X_{jl}^*R^{(j)}_{k+n,l+n},\quad
\varepsilon_j^{(2)}=\frac1p\sum_{k=1}^p(|X_{j,k}|^2-1)R^{(j)}_{k+n,k+n}\notag\\
\varepsilon_j^{(3)}&=\frac1p\sum_{k=1}^pR^{(j)}_{k+n,k+n}-\frac1p\sum_{k=1}^pR_{k+n,k+n},\quad
\varepsilon_j^{(4)}=\frac{1}p\sum_{k=1}^pR_{k+n,k+n}-\frac{1}p\E\left(\sum_{k=1}^pR_{k+n,k+n}\right).\notag
\end{align}

This implies that
\begin{equation}\label{qq}
\widetilde s_p(z)=-\frac1{z+y\widetilde s_p(z)+\frac{y-1}z}+\delta_p(z),
\end{equation}
where
\begin{equation}
 \delta_p(z)=\frac1{n\;(z+y\widetilde s_p(z)+\frac{y-1}z)}\sum_{j=1}^n\varepsilon_jR_{jj}.
\end{equation}

Throughout this paper we shall consider 
$z=u+iv$ with $a\le |u|\le b$ and 
\newline$0<v<C$.

The main result of this Section is
\begin{lem}\label{Lemma 3.4} Let
$$
\text{\rm{Im}}\,\Big\{y\delta_p(z)+z +\frac{y-1}z\Big\}\ge 0.
$$
Then 
$$
\left|z+\frac{y-1}z+ys_p(z)\right|\ge 1.
$$
\end{lem}
\begin{proof} From representation (\ref{qq}) it follows that
\begin{equation}
 \im\left\{ys_p(z)+z+\frac{y-1}z\right\}=\frac{\im\left\{ys_p(z)+z+\frac{y-1}z\right\}}
{|ys_p(z)+z+\frac{y-1}z|^2}+\im\{\delta_p(z)+z+\frac{y-1}z\}.
\end{equation}
This equality concludes the proof.
\end{proof}
\section{Bounds for $\delta_p(z)$}
We start from the simple bound for the $\delta_p(z)$.
\begin{lem}\label{Lem3.0}Under the conditions of Theorem \ref{thm1.1} the following bound holds for $1\ge v\ge CM^{1/2}n^{-1/2}$
 \begin{equation}
  |\delta_p(z)|\le \frac 1{|z+y\widetilde s_p(z)+\frac{y-1}z|^2}\frac C{nv^4}.
 \end{equation}
\end{lem}
\begin{proof}
 Note that 
\begin{equation}
 |\delta_p(z)|\le \frac 1{|z+y\widetilde s_p(z)+\frac{y-1}z|^2}(\frac1n\sum_{j=1}^n|\E \varepsilon_{j}|+\frac1n\sum_{j=1}^n\E \varepsilon_{j}|^2|R_{j,j}|).
\end{equation}
Using inequalities (\ref{ineq01}), (\ref{ineq01a}), (\ref{ineq01b}), and (\ref{ineq01c}) below and inequality $|R_{j,j}|\le 1/v$, we get
\begin{align}
 |\delta_p(z)|&\le \frac 1{|z+y\widetilde s_p(z)+\frac{y-1}z|^2}(\frac1{nv}+\frac1{nv}\sum_{j=1}^n\E |\varepsilon_{(j)}|^2\notag\\ &\le \frac 1{|z+y\widetilde s_p(z)+\frac{y-1}z|^2}(\frac1{nv}+
\frac C{nv^3})
\end{align}
Thus the Lemma is proved.
\end{proof}

In this Section we give  bounds for remainder term $\delta_p(z)$ in the equation (\ref{qq}).We first start with  bounds assuming that there exist positive constants $a_1$, $a_2$ such that
\begin{equation}\label{con5.2}
a_1\le\left|z+\frac{y-1}z+ys_p(z)\right|\le a_2.
\end{equation}
\begin{lem}\label{lem:4.1}
There exists a positive absolute constant $C$ such that,  for $v\ge cn^{-1}$ with some other positive absolute constant $c$,
\begin{equation}\label{ineq01}
 \E|\varepsilon_j^{(1)}|^2\le \frac{C(1+|s_p(z)|)}{nv}
\end{equation}
\begin{equation}\label{ineq01a}
 \E|\varepsilon_j^{(2)}|^2\le \frac{C(1+|s_p(z)|)}{nv}
\end{equation}
and
\begin{equation}\label{ineq02}
 \E|\varepsilon_j^{(1)}|^4\le \frac{CM_4^2(1+|\widetilde s_p(z)|)}{n^2v^2}.
\end{equation}
\end{lem}
\begin{proof}
Consider inequality (\ref{ineq01}). We have
\begin{equation}\label{ineq4.5}
\E|\varepsilon_j^{(1)}|^2\le\frac2{p^2}\sum_{k,l=1}^p\E|R^{(j)}_{k,l}|^2\le \frac1{p^2}\E\Tr\bold R^{(j)}(\bold R^{(j)})^*\le \frac2{p^2v}\E\im\Tr\bold R^{(j)}.
\end{equation}

Applying Lemma \ref{Lemma 3.2}, we get
\begin{equation}
 |\Tr\bold R-\Tr\bold R^{(j)}|\le 1/v.
\end{equation}Note that
\begin{equation}
 \frac1{2n}\E\im\Tr \bold R(z)\le (1+y)|\widetilde s_p(z)|+\left|\im\left\{\frac{1-y}{z}\right\}\right|.
\end{equation}
It is straighforward to check that
\begin{equation}
 \left|\im\left\{\frac{1-y}{z}\right\}\right|\le 1
\end{equation}
The last inequalities together conclude the proof of inequality (\ref{ineq01}).
The proof of inequality (\ref{ineq01a}) is similar.
Furthermore,
\begin{equation}
 \E|\varepsilon_j^{(1)}|^4\le\frac {CM_4^2}{p^4}\E\left(\sum_{k,l=1}^p|R^{(j)}_{k,l}|^2\right)^2
\le\frac {CM_4^2}{p^2v^2}\E\left(\frac1p\,\im\Tr\bold R^{(j)}\right)^2.
\end{equation}
Similar to inequality (\ref{ineq01}) we get
\begin{equation}
 \E|\varepsilon_j^{(1)}|^4\le \frac {CM_4^2(1+|\widetilde s_p(z)|)^2}{p^2v^2}
\end{equation}
Thus the Lemma is proved.
\end{proof}
\begin{lem}\label{lem4.1}
For any $j=1,\ldots, n$ the following inequality
\begin{equation}\label{ineq01b}
 |\varepsilon_j^{(3)}|\le \frac1{nv}
\end{equation}
holds.
\end{lem}
\begin{proof}
 The result follows immediately from Lemma \ref{Lemma 3.2} with $\bold A= \bold H$. 
\end{proof}

\begin{lem}\label{lem:4.2}The follwoing bound holds for all $v>0$
\begin{equation}\label{ineq01c}
 \E|\varepsilon_j^{(4)}|^2\le \frac 4{nv^2}.
\end{equation}

 There exist  positive constants $c$ and $C$ depending on $a_1$ and $a_2$ such that for any $v\ge cn^{-\frac12}$
\begin{equation}\label{ineq10}
 \E|\varepsilon_j^{(4)}|^2\le \frac {CM_4(1+|\widetilde s_p(z)|)}{n^2v^3}
\end{equation}
and
\begin{equation}\label{110}
 \E|\varepsilon_j^{(4)}|^3\le \frac {CM_4(1+|\widetilde s_p(z)|)}{n^{\frac52}v^4}
\end{equation}
and
\begin{equation}\label{ineq11}
 \E|\varepsilon_j^{(4)}|^4\le \frac {CM_4(1+|\widetilde s_p(z)|)}{n^3v^5}.
\end{equation}
\end{lem}
\begin{proof}Note that
\begin{equation}
 \varepsilon_j^{(4)}=\frac1p(\sum_{j=1}^pR_{j+n,j+n}-\E\sum_{j=1}^pR_{j+n,j+n})=
\frac1p(\Tr\bold R(z)-\E\Tr\bold R(z))
\end{equation}

 Let $\E_k$ denote the conditional expectation given $X_{lm},\ 1\le l\le k; \ 1\le m\le p$.
\begin{equation}\label{ineq1}
 \E|\varepsilon_j^{(4)}|^2=\frac1{p^2}\sum_{k=1}^n\E|\gamma_k|^2,
\end{equation}
where
\begin{equation}
 \gamma_k=\E_{k}(\Tr\bold R)-\E_{k-1}(\Tr\bold R).
\end{equation}
Since $\E_{k}\Tr\bold R^{(k)}=\E_{k-1}\Tr\bold R^{(k)}$ we have
\begin{equation}
 \gamma_k=\E_{k}\sigma_k-\E_{k-1}\sigma_k,
\end{equation}
where
\begin{equation}
\sigma_k=(\Tr\bold R-\Tr\bold R^{(k)}) .
\end{equation}
According to Lemma \ref{Lemma 3.2}, we may represent $\sigma_k$ as follows
\begin{equation}
 \sigma_k=\sigma_k^{(1)}+\sigma_k^{(2)}+\sigma_k^{(3)}+\sigma_k^{(4)},
\end{equation}
where
\begin{align}
\sigma_k^{(1)}&=\frac{1+\frac1p\sum_{r=1}^n\sum_{s=1}^pX_{kr}\overline X_{ks}(\bold R^{(k)})^2_{rs}}{z+y\widetilde s_p(z)+\frac{y-1}z}\notag\\ 
\sigma_k^{(2)}&=\frac{\varepsilon_k\sigma_k}{z+y\widetilde s_p(z)+\frac{y-1}z}\notag\\ 
\sigma_k^{(3)}&=\frac{\frac1p\left(\sum_{r=1}^n\sum_{s=1}^pX_{kr}\overline X_{ks}(\bold R^{(k)})^2_{rs}-\Tr (\bold R^{(k)})^2\right)}{z+y\widetilde s_p(z)+\frac{y-1}z}\notag.
\end{align}
Since
\begin{equation}
 \E_{k}\sigma_k^{(1)}=\E_{k-1}\sigma_k^{(1)},
\end{equation}
we get
\begin{equation}\label{ineq4}
 \E|\gamma_k|^2\le 2(\E|\sigma_k^{(2)}|^2+\E|\sigma_k^{(3)}|^2)\le C(\frac 1{v^2}
\E|\varepsilon_k|^2+\E|\sigma_k^{(3)}|^2).
\end{equation}
By definition of $\varepsilon_k$, we have
\begin{equation}
 \E|\varepsilon_k|^2\le 4\E|\varepsilon_k^{(1)}|^2+4\E|\varepsilon_k^{(2)}|^2+4\E|\varepsilon_k^{(3)}|^2
+4\E|\varepsilon_k^{(4)}|^2.
\end{equation}

According to Lemmas \ref{lem:4.1} -- \ref{lem:4.2}, we have
\begin{equation}\label{ineq5}
\E|\varepsilon_k|^2\le \frac{C(1+|\widetilde s_p(z)|)}{nv}+4\E|\varepsilon_k^{(4)}|^2.
\end{equation}

Furthermore,
\begin{equation}\label{ineq6}
 \E|\sigma_k^{(3)}|^2\le \frac C{n^2v^3}\im\Tr \bold R^{(k)}\le \frac{C(1+|\widetilde s_p(z)|)}{nv^3}.
\end{equation}
Inequalities (\ref{ineq4}), (\ref{ineq5}) and (\ref{ineq6}) together imply that
\begin{equation}\label{ineq7}
 \E|\gamma_k|^2\le \frac{C(1+|\widetilde s_p(z)|)}{nv^3}+\frac C{v^2}\E|\varepsilon_k^{(4)}|^2
\end{equation}
From the inequalities (\ref{ineq1}) and (\ref{ineq7}) it follwos that
\begin{equation}
 \E|\varepsilon_k^{(4)}|^2\le\frac{C(1+|\widetilde s_p(z)|)}{n^2v^3}+\frac C{nv^2}\E|\varepsilon_k^{(4)}|^2.
\end{equation}
For $v\ge cn^{-\frac12}$ with some sufficiently small positive absolute constant $c$, we get
\begin{equation}
 \E|\varepsilon_k^{(4)}|^2\le\frac{C(1+|\widetilde s_p(z)|)}{n^2v^3}.
\end{equation}
Thus the inequality (\ref{ineq10}) is proved.
To prove inequality (\ref{ineq11}) we use the Burkholder inequality for martingales
(see Hall and Heyde \cite{hall}, p.24).
We get
\begin{equation}\label{ineq12}
\E|\varepsilon_k^{(4)}|^4\le \frac n{p^4}\sum_{l=1}^n\E|\gamma_l|^4.
\end{equation}
Using that $|\gamma_l|\le \frac2v$, we get
\begin{equation}\label{ineq13}
 \E|\gamma_l|^4\le \frac2{v^2}\E|\gamma_l|^2\le \frac{CM_4(1+|\widetilde s_p(z)|^4)}{nv^5}.
\end{equation}
Inequalities (\ref{ineq12}) and (\ref{ineq13}) together imply that
\begin{equation}
\E|\varepsilon_k^{(4)}|^4\le  \frac{CM_4(1+|\widetilde s_p(z)|^4)}{n^3v^5}. 
\end{equation}
Thus the Lemma is proved.
\end{proof}
\begin{lem}\label{Lemma 5.4}
 There exist some positive constants $c$ and $C$ such that, for any $1\ge v\ge cn^{-\frac12}$, the following inequality holds
\begin{equation}
 \frac1n\sum_{j=1}^n\E|R_{k,k}|^2\le C.
\end{equation} 
\end{lem}
\begin{proof}
 To prove this Lemma we repeat the proof of Lemma 5.4 in \cite{GT:03}.
Let
\begin{equation}
 U^2=\frac1n\sum_{j=1}^{n+p}\E|R_{k,k}|^2.
\end{equation}

By equality (\ref{ineq3.1}), we have
\begin{equation}\label{ineq:011}
U^2\le C(1+\frac1n\sum_{j=1}^n\E|\varepsilon_{j}|^2|R_{j,j}|^2).
\end{equation}
Applying Lemmas \ref{lem:4.1}--\ref{lem:4.2}, we obtain
\begin{equation}\label{ineq:012}
\frac1n\sum_{j=1}^n\E|\varepsilon_{j}^{(1)}|^2|R_{j,j}|^2\le \frac{CM_4}{nv^2}\left(\frac1n\sum_{j=1}^n\E|R_{j,j}|^2\right)^{\frac12}.
\end{equation}
Furthermore,
\begin{equation}\label{ineq:013}
 \frac1n\sum_{j=1}^n\E|\varepsilon_{j}^{(3)}|^2|R_{j,j}|^2\le\frac{C}{n^2v^4}.
\end{equation}
To bound $\frac1n\sum_{j=1}^n\E|\varepsilon_{j}^{(4)}|^2|R_{j,j}|^2$ we use that $\varepsilon_j^{(4)}$ does not  depend on $j$. We write
\begin{align}\label{ineq15}
\frac 1n\sum_{j=1}^n\E|\varepsilon_{j}^{(4)}|^2|R_{j,j}|^2&=
\E|\varepsilon_{1}^{(4)}|^2\left(\frac 1n\sum_{j=1}^n |R_{j,j}|^2\right)\notag\\
&\le
\frac Cv\E|\varepsilon_{1}^{(4)}|^2\frac1n\im\Tr\bold R(z)\notag\\
&\le \frac {C|\widetilde s_p(z)|}{v}
\E|\varepsilon_{1}^{(4)}|^2+\frac Cv\E|\varepsilon_{1}^{(4)}|^2|\frac 1n(\Tr \bold R(z)-
\E\Tr\bold R(z)|\notag\\&\le\frac {C(1+|\widetilde s_p(z)|)}{v}
\E|\varepsilon_{1}^{(4)}|^2+ \frac Cv\E|\varepsilon_{1}^{(4)}|^3
\end{align}
Inequalities (\ref{ineq10}), (\ref{ineq11}), and (\ref{ineq15}) together imply
\begin{equation}\label{ineq16}
 \frac1n\sum_{j=1}^n\E|\varepsilon_{j}^{(4)}|^2|R_{j,j}|^2\le\frac {CM_4(1+|\widetilde s_p(z)|)}{n^2v^4}
+\frac {CM_4(1+|\widetilde s_p(z)|)}{\sqrt{n^5v^{10}}}.
\end{equation}
Let 
\begin{equation}
 T:=\frac1n\sum_{j=1}^{n+p}\E|\varepsilon_{j}^{(2)}|^2|R_{j,j}|^2.
\end{equation}
From inequalities (\ref{ineq:011}), (\ref{ineq:012}), (\ref{ineq:013}), and (\ref{ineq16}) it follows that, for $v\ge cn^{-\frac12}$,
\begin{equation}
 U^2\le C+\delta U+T.
\end{equation}
Solving this equation with respect to $U$, we get
\begin{equation}\label{ineq:5.58}
 U^2\le C+T.
\end{equation}
To bound $T$ we start from the obvious inequality
\begin{equation}\label{4.41}
T\le \frac1{v^2}\frac1n\sum_{j=1}^{n+p}\E|\varepsilon_{j}^{(2)}|^2\le
 \frac C{nv^2}\frac1n\sum_{j=1}^{n+p}\left(\frac1n{\sum}^{(j)}\E|R^{(j)}_{k,k}|^2\right),
\end{equation}
where ${\sum}^{(j)}$ denotes the sum over all $k=1,\ldots, n+p$ except $k=j$.
Introduce now some integer number $m=m(n)$ depending on $n$ such that
\newline
$mv^{-1}\le a_1/4$. Without loss of generality we may assume that
$m\le n/2$. Since \newline
$|\widetilde s_{p-l}(z)-\widetilde s_{p-l-1}(z)|\le \frac{1}{n-l}$
we get
$$
a_1/2\le \min_{1\le l\le m}|\widetilde s_{p-l}(z)+z+\frac{y-1}{z}|\le
\max_{1\le l\le m} |y\widetilde s_{p-l}(z)+z+\frac{y-1}{z}|\le \frac32 a_2.
$$
Let $\bold j^{(r)}=(j_1,\ldots,j_r)$ with $1\le j_1\ne j_2\ldots\ne j_r\le n$,
$r=1,\ldots ,m$.
Denote by $\bold H^{(\bold j^{(r)})}$ the matrix which is obtained from $\bold H$ by deleting
the $j_1$th, $\ldots$, $j_r$th rows and columns, and let 
$$
\bold R^{(\bold j^{(r)})}=
\left(\frac1{\sqrt{n-r}}\bold H^{(\bold j^{(r)})}-z\bold I_{n+p-r}\right)^{-1}.
$$
Arguing similar as in  inequality (\ref{4.41}) we get that uniformly for
 $r=1,\ldots,m-1$, and for $v\ge C_1(a_1,a_2)n^{-\frac12}M^{\frac12}$
\begin{align}\label{5.45}
\frac1{n}\sum_{k=1,\,
k\notin \bold j^{(r)}}^n\E|R^{(\bold j^{(r)})}_{k,k}|^2
&\le \frac{C_0(a_1,a_2)M}{nv^2}
\Big(\frac1n\sum_{k=1,\,k\notin \bold j^{(r)}}^n
\Big(\frac1n\sum_{j=1,j\notin \bold j^{(r+1)}}^n
\E|R^{(\bold j^{(r+1))}}_{j,j}|^2\Big)\Big)\notag
\\&\qquad\qquad\qquad\qquad\qquad\qquad\qquad\qquad+C_0(a_1,a_2).
\end{align}
Note that the constants $C_0(a_1,a_2)$ and $C_1(a_1,a_2)$
do not depend on 
$l=1,\ldots,m$. 

Applying inequality (\ref{5.45}) recursively we get for 
$1\ge v\ge C_1(a_1,a_2)n^{-1/2}M^{\frac12}$,
\begin{align}\label{5.46}
\frac1n\sum_{k=1}^n\E|R_{k,k}|^2&\le C_0(a_1,a_2)\sum_{r=0}^{m-1}
\Big(\frac{C_0(a_1,a_2)M}{nv^2}\Big)^r\notag\\+
\Big(\frac{C_0(a_1,a_2)M}{nv^2}&\Big)^m\Big(\frac1n\sum_{k=1,\, k\notin \bold j^{(m-1)}}^n
\Big(\frac1n\sum_{j=1,\,j\notin \bold j^{(m)}}^n
\E|R^{\bold j^{(m)}}_{j,j}|^2\Big)\Big)
 \end{align}
Without loss of generality we may assume that 
$$
\frac{C_0(a_1,a_2)M}{nv^2}\le \frac12.
$$
Similar to inequality (\ref{ineq4.5})
we get that
\begin{equation}\label{5.47}
\frac1n\sum_{j=1,\, j\notin \bold j^{(m)}}^n
\E|R_{\bold j^{(m)}}(j,j)|^2\le \E\Tr|R_{\bold j^{(m)}}|^2
\le\frac {C_0(a_1,a_2)}v.
\end{equation}
The inequalities (\ref{5.46}) and (\ref{5.47}) together imply that
\begin{equation}\label{5.48}
\frac1n\sum_{k=1}^n\E|R(k,k)|^2\le 2C_0(a_1,a_2)+
\frac1{2^m}\frac Cv.
\end{equation}

Choosing $m=[C\log n]$ such that $2^{-m}\le Cv$ concludes 
the proof. 
\end{proof}
\begin{lem}\label{Lemma 5.5}Assume that condition $(\ref{con5.2})$ holds. Then there exist positive 
constants $C_3(a_1,a_2)$ and $C_4(a_1,a_2)$ such that for $v\ge C_3(a_1,a_2)
n^{-1/2}M^{1/2}$ the following inequality holds
$$
|\delta_p(z)|\le \frac{C_4(a_1,a_2)M}{nv}.
$$
\end{lem}
\begin{proof}The equalities (4.5) and (4.6) imply that
\begin{equation}\label{5.49}
|\delta_p(z)|\le 
\frac{C}{|z+y\widetilde s_p(z)+\frac{y-1}z|^2}\Bigl(\frac1p\sum_{k=1}^{n+p}|\E\varepsilon_k|+
\frac1p\sum_{k=1}^{n+p}\E|\varepsilon_k|^2|R(j,j)|\Bigr).
\end{equation}
According to Lemma \ref{lem4.1} and inequality (\ref{con5.2}) we get
\begin{equation}\label{5.50}
\frac{C}{|z+ys_n(z)+\frac{y-1}z|^2}\Bigl(\frac1n\sum_{k=1}^n|\E\varepsilon_k|\Bigr)\le
\frac{C}{nva_1^{2}}\le
\frac{C(a_1,a_2)}{nv}.
\end{equation}
Using the representation (\ref{rep5.1}), we obtain
\begin{equation}\label{5.51}
\frac{C}{|z+ys_n(z)+\frac{y-1}z|^2}
\Bigl(\frac1n\sum_{k=1}^n\E|\varepsilon_k|^2|R(j,j)|\Bigr)\le
C(a_1,a_2)\sum_{\nu=1}^4
\Bigl(\frac1n\sum_{k=1}^n\E|\varepsilon_k^{(\nu)}|^2|R(j,j)|\Bigr).
\end{equation}

Similar to inequality (\ref{5.46}) and by Lemma \ref{Lemma 3.4} we arrive at
\begin{align}\label{5.52}
\frac1n\sum_{k=1}^n\E|\varepsilon_k^{((1)}|^2|R(k,k|&\le
 \Big(\frac1n\sum_{k=1}\E|\varepsilon_k^{(1)}|^4\Big)^{1/2}
\Big(\frac1n\sum_{k=1}^n\E|R(k,k)|^2\Big)^{1/2}\\&\le
\frac {C(a_1,a_2)M^{\frac12}}{nv}.
\end{align}
By Lemma \ref{lem4.1}, $|\varepsilon_k^{(3)}|\le (nv)^{-1}$   we have
\begin{equation}\label{5.53}
\frac1n\sum_{k=1}^n\E|\varepsilon_k^{(3)}|^2|R_{k,k}|\le 
\frac{1}{n^2v^3}
\le \frac{C(a_1,a_2)}{nv}.
\end{equation}

Finally, note that
$$
\frac1n\sum_{k=1}^n\E|\varepsilon_k^{(2)}|^2|R(k,k|\le
\frac1{nv}\sum_{k=1}^n\E|\varepsilon_k^{(2)}|^2\le
\frac{C(a_1,a_2)M}{nv}\Big(\frac1n\sum_{j=1,j\ne k}\E|R_(j,j)^{(k)}|^2\Big).
$$
Applying Lemma \ref{Lemma 5.4} to the matrix $\bold H^{(k)}$ we get
\begin{equation}\label{5.55}
\frac1n\sum_{k=1}^n\E|\varepsilon_k^{(2)}|^2|R(k,k|\le\frac{C(a_1,a_2)M}{nv}.
\end{equation}

The inequalities (\ref{5.49})--(\ref{5.55}) together imply that for 
$1\ge v\ge C_1(a_1,a_2)n^{-1/2}M^{\frac12}$
$$
|\delta_n(z)|\le \frac{C(a_1,a_2)M}{nv},
$$
which proves Lemma \ref{Lemma 5.5}.
\end{proof}
\begin{lem}\label{lem5.1} Assuming the  conditions of Theorem \ref{thm1.1}, there exists an absolute positive constant $C$ such that for any $1\ge v\ge CM^{1/2}n^{-1/2}$ and $u\in [a,b]$, the following inequality holds
 \begin{equation}\label{ineq5.0}
  \im\Bigl\{z+y\widetilde s_p(z)+\frac{y-1}z\Bigr\}>0,\quad z=u+iv.
 \end{equation}

\end{lem}

\begin{Proof of} Lemma \ref{lem5.1}.
Assume that for $r_n(z):=z+y\delta_p(z)+\frac{y-1}z$ the following equality holds
\begin{equation}\label{5.56}
\im\bigl\{r_n(z)\bigr\}=0.
\end{equation}
Denote be $t(z):=y\widetilde s_p(z)+\frac{y-1}z+z$. 
Since 
$$
t(z)=-\frac y{t(z)}+r_n(z)
$$ 
this immediately implies that
$$
\im t(z)=-\im\Bigl\{\frac y{t(z)}\Bigr\}.
$$
 Since $\im\{t(z)\}\ge\im z=v>0$ 
this implies that
$$
|t(z)|=\sqrt y.
$$ 
Hence  condition (\ref{con5.2}) holds with
$a_1=a_2=\sqrt y$ and we have
$$
|\delta_p(z)|\le \frac{CM}{nv}.
$$
Then for any $v\ge 2n^{-\frac12}\sqrt{CM}$,
$$
|\delta_n(z)|\le \frac14 v<v,
$$
holds.
But condition (\ref{5.56}) implies that
$$
|\delta_p(z)|\ge v,
$$
which is a contradiction. 
Hence we conclude that $\im\{z+y\delta_p(z)+\frac{y-1}z\}\ne0$ in the region
$v\ge 2n^{-\frac12}\sqrt{CM}$. From Lemma \ref{Lem3.0}
it follows for example
that, for $v=1$,
$\im\{r_n(z)\}>0$.  Since the function 
$\im\{r_n(z)\}$ is continuous in the region 
$v\ge C_1n^{-\frac12}\sqrt{M}$
we get that
$\im\{r_n(z)\}>0$ for $v\ge C_1n^{-\frac12}\sqrt{M}$.
This proves Lemma \ref{lem5.1}.
\end{Proof of}

\begin{Proof of} Theorem \ref{thm1.1}. Recall that $1\ge y\ge \theta>0$. Let 
$v_0=\max\{\gamma_0\Delta_p,2n^{-\frac12}C_1M^{\frac12}\}$
with a $\gamma_0$ such that $1>\gamma_0>0$ to be chosen later. 
By Lemma \ref{lem5.1} for any $1\ge v\ge v_0$ we have
$$
\im\{z+y\delta_p(z)+\frac{y-1}z\}>0. 
$$
Note that the constant $C_1$ does  not depend on $\gamma_0$.
In addition we have
\begin{align}
|\widetilde s_p(z)-\widetilde s_y(z)|&=
\Bigl|\int_{-\infty}^{\infty}\frac1{x-z}d\Bigl(\E\,\widetilde F_p(x)-\widetilde F_y(x)\Bigr)\Bigr|
\\&=\Bigl|\int_{-\infty}^{\infty}\frac{\E\,\widetilde F_p(x)-\widetilde F_y(x)}{(x-z)^2}dx\Bigr|\le
\frac{\Delta_p}v\le \frac 1\gamma_0.
\end{align}
This implies that for $z=u+iv$ such that $|u|\in[a,b]$,
$1\ge v\ge v_0$, we have
\begin{equation}\label{5.57}
|y\widetilde s_p(z)+z+\frac{y-1}z|\le \frac 1\gamma_0+5.
\end{equation}
From equality (\ref{qq}) it follows that
\begin{equation}\label{q0}
s_p(z)=-\frac1{2y}\Bigl(z+\frac{y-1}z-y\delta_p(z)-\sqrt{(z+\frac{y-1}z+y\delta_p(z))^2-4y}\Bigr).
\end{equation}
Introduce the function
\begin{equation}\label{q1}
 q(z):=-\frac1{2y}(z-\sqrt{z^2-4y}).
\end{equation}
Equalities (\ref{q0}) and (\ref{q1})  together imply that
for $v\ge v_0$
\begin{equation}
z+y\widetilde s_p(z)+\frac{y-1}z=q(\omega+y\delta_p(z))\label{5.58}
\end{equation}
where $\omega:=z+\frac{y-1}z$.
Let $s(z)$ denote the Stieltjes transform of the  semicircular law.
Then
$q(z)=\frac1{sqrt y}s(z/\sqrt y)$. This implies in particular that $|q(z)|\le 1/\sqrt y$.
Since $\im\{y\delta_p(z)+\omega\}>0$  the equality (\ref{5.58})  immediately implies that
\begin{equation}
|z+y\widetilde s_p(z)+\frac{y-1}z|\ge 1/\sqrt y,\qquad \text{for}\qquad v\ge v_0\label{5.59}
\end{equation}
From the inequalities (\ref{5.58}) and (\ref{5.59}) it follows 
that condition (\ref{con5.2}) holds with
$a_1=1$,
 and  $a_2=\frac1{\gamma_0}+5$.
The relation (\ref{5.58}) implies that
\begin{equation}
|\widetilde s_p(z)-\widetilde s_y(z)|\le \frac1{\sqrt y}\left|{q(\omega)}-{q(\omega+y\delta_p(z)}\right|.\label{5.60}
\end{equation}
After a simple calculation we get
\begin{equation}
|\widetilde s_p(z)-\widetilde s_y(z)|\le \frac{y|\delta_n(z)|}{|\sqrt{(\omega+y\delta_p(z))^2-4y}+\sqrt{\omega^2-4y}}.\label{5.61}
\end{equation}
By Lemma \ref{Lemma 5.5} we obtain for $1\ge v\ge v_0$, 
\begin{equation}
|\delta_n(z)|\le \frac14v,\label{5.62}
\end{equation}
and for $z=u+iv$ such that $u\in I$ we get
\begin{equation}
\min\{\sqrt{|\omega^2-4y|},\sqrt{|(\omega+y\delta_n(z))^2-4y|}\}\ge C\sqrt{v}.\label{5.63}
\end{equation}
Inequalities (5.61)--(5.63) imply that for $z=u+iv$ such that 
$u\in I$ and $1\ge v\ge v_0$ 
\begin{equation}
|\widetilde s_p(z)-\widetilde s_y(z)|\le \frac {C|\delta_p(z)|}{\sqrt v}.\label{5.64}
\end{equation}

By Lemma \ref{Lemma 5.5} we have
\begin{equation}
|\delta_p(z)|\le \frac{C(\gamma_0)M}{nv}.\label{5.65}
\end{equation}
From (5.64) and (5.65) it follows that
$$
|\widetilde s_p(z)-\widetilde s(z)|\le \frac{C(\gamma_0)M}{nv^{\frac32}}.
$$
Choosing in Corollary 2.3 $V=1$ and using the inequality
(4.29) we get after integrating in $u$ and $v$
$$
\Delta_n \le C_1Mn^{-1}+C_2v_0+C_3(\gamma_0)Mn^{-1}v_0^{-1}.
$$
Since $v_0\ge 2n^{-\frac12}\sqrt{C_1(\gamma_0)M}$ we get
$$
\Delta_n\le C(\gamma_0)M^{\frac12}n^{-\frac12}+C_3v_0
$$
Recall that $C_2$ does not depend on $\gamma_0$.
If $v_0=2n^{-\frac12}C_1(\gamma_0)M^{\frac12}$ then
$$
\Delta_n\le C(\gamma_0)M^{\frac12}n^{-\frac12}.
$$
We choose $\gamma_0=\frac1{2C_3}$.
If $v_0=\gamma_0\Delta_n$ then 
$$
\Delta_n\le C(\gamma_0)M^{\frac12}(1-C_3\gamma_0)^{-1}n^{-\frac12}\le
2C(\gamma_0)M^{\frac12}n^{-\frac12}.
$$
This completes the proof of Theorem \ref{thm1.1}. 
\end{Proof of}

{\bf Acknowledgment.}
The authors would like to thank Dmitry Timushev for careful reading of the manuscript.

%


\end{document}